\documentclass[10 pt, reqno]{amsart}

\usepackage{mathtools}

\usepackage[dvipsnames]{xcolor}
\usepackage[a4paper,margin=1in]{geometry}
\usepackage{latexsym, mathrsfs, enumerate, fullpage,graphicx,amsmath,amsfonts,amssymb,amsthm,tikz-cd}

\usepackage{hyperref}
\newtheorem{theorem}{Theorem}[section]

\newtheorem{lemma}[theorem]{Lemma}

\newtheorem{corollary}[theorem]{Corollary}

\newtheorem{defn}{Definition}[section]

\begin{document}
\title{A Cheng-type Eigenvalue-Comparison Theorem for the Hodge Laplacian}
\author[A. Bhattacharya and S. Maity]{Anusha Bhattacharya and Soma Maity}

\address{Department of Mathematical Sciences, Indian Institute of Science Education and Research Mohali, \newline Sector 81, SAS Nagar, Punjab- 140306, India.}
\email{anushabhattacharya20@gmail.com}

\address{Department of Mathematical Sciences, Indian Institute of Science Education and Research Mohali, \newline Sector 81, SAS Nagar, Punjab- 140306, India.}
\email{somamaity@iisermohali.ac.in}

\begin{abstract}
We consider the class of closed Riemannian $n$-manifolds with Ricci curvature and injectivity radius bounded below by uniform constants, and an upper bound on the diameter. We establish a uniform upper bound for the eigenvalues of the Hodge  Laplacian and the connection Laplacian acting on differential forms on Riemannian manifolds in this class, similar to Cheng's eigenvalue comparison theorem for the Laplace-Beltrami operator acting on smooth functions. This extends earlier work of  Dodziuk and Lott, which required sectional curvature bounds in addition to bounds on other geometric quantities. As an application, we derive a uniform upper bound on the spectral gap for non-compact manifolds with Ricci curvature bounded below.

\end{abstract}


\maketitle
\section{Introduction} A fundamental problem in spectral geometry is to understand how bounds on the geometric properties of a Riemannian manifold can be used to derive estimates on the spectrum of the Laplacian. One of the classical results in this direction is the eigenvalue comparison theorem proved by Cheng\cite{Cheng}, which gives an upper bound for the eigenvalues of the Laplace-Beltrami operator in terms of the simply connected constant curvature spaces. 
We state it as follows.
\begin{theorem}\cite{Cheng} \label{thm:cheng} Suppose $M$ is a closed $n$-dimensional Riemannian manifold with diameter $D$. Let $\lambda_k(M)$  denote the $k$-th  eigenvalue of the Laplace-Beltrami operator (non-negative), counted with multiplicity. If Ricci curvature of $M$ is greater than or equal to $(n-1)\xi$ then
$$\lambda_k(M)\leq \lambda_0^D\left(B_\xi\left(\frac{D}{2k}\right)\right)$$
where $\lambda_0^D(B_{\xi}\left(r)\right)$ denotes the first positive Dirichlet eigenvalue of the Laplacian on a ball of radius $r$ in the simply connected $n$-dimensional space of constant sectional curvature $\xi$. 
\end{theorem}
This naturally leads to the question of whether such eigenvalue estimates can be obtained for the Laplacian acting on differential forms.  For  $1\le p\le n$, the Hodge Laplacian $\Delta$ acting on differential forms of degree $p$ is defined as 
\[\Delta=dd^*+d^*d\] where $d$ is the exterior derivative and  $d^*=(-1)^{n(p+1)+1}*d*$, where $*$ is the Hodge operator. Let $\lambda_{k,p}(M)$ denote the $k$-th positive eigenvalue of $\Delta$ acting on differential forms of degree $p$ counted with multiplicity. Under sectional curvature bounds, Dodziuk \cite{ dodziuk1982} and Lott \cite{LottJ} have obtained upper bounds for $\lambda_{k,p}(M)$.  In this paper, we derive explicit upper bounds for the eigenvalues of the Hodge Laplacian under the weaker assumption of a lower Ricci curvature bound on the manifold. In \cite{lott2022}, Lott observed that $\lambda_{k,p}(M)$ is uniformly bounded in the class of manifolds with Ricci curvature and injectivity radius bounded below, and diameter bounded above. This follows from the pre-compactness of this class in the  $C^\alpha$-topology, for $0<\alpha<1$, established by Anderson and Cheeger\cite{CheegerAnderson}, together with the continuity of $\lambda_{k,p}(M)$ in this topology. Here, we make this qualitative uniform bound explicit by obtaining a Cheng-type upper bound for $\lambda_{k,p}(M)$ for manifolds with  Ricci curvature bounded below. 
\par To derive an explicit estimate of $\lambda_{k,p}(M)$,  we require quantitative local control of the metric coefficients that is provided by harmonic coordinates. We now recall the definition of $C^{m,\alpha}$-harmonic radius on $(M,g)$.
\begin{defn}{\cite{Hebey}} \label{def:harmC}Let $(M,g)$ be a closed Riemannian $n$-manifold and let $x\in M$. Given $Q>1$ and $0<\alpha<1$, the $C^{m,\alpha}$-harmonic radius at $x$, denoted by $r_H=r_H(Q,m,\alpha)(x)$, is  the largest number such that on the geodesic ball $B(x,r_H)$, there exists a $C^{m,\alpha}$-coordinate chart $(y^1,y^2,\cdots, y^n)$ centered at $x$ such that $\Delta y^i=0$, for all $i$. If $g_{ij}$ denote the components of the metric $g$ in these coordinates, then
\begin{enumerate}
    \item
    \[
        Q^{-1}\delta_{ij} \leq g_{ij} \leq Q\delta_{ij}
        \qquad \text{as bilinear forms},
    \]
    
    \item
    \[
        \sum_{1\leq |\beta|\leq m}
        r_H^{|\beta|}
        \sup_x |\partial^\beta g_{ij}(x)|
        +
        \sum_{|\beta|=m} r_H^{m+\alpha}
        \sup_{y\neq z}
        \frac{
            |\partial^\beta g_{ij}(y)-\partial^\beta g_{ij}(z)|
        }{
            d(y,z)^\alpha
        }
        \leq Q-1.
    \]
\end{enumerate}
 The harmonic radius $r_H(Q,m,\alpha)(M)$ of $(M,g)$ is defined by
\[
    r_H(Q,m,\alpha)(M)
    =
    \inf_{x\in M} r_H(Q,m,\alpha)(x).
\]
\end{defn} On a class of closed manifolds with Ricci curvature and injectivity radius bounded below, Anderson and Cheeger showed the existence of a uniform lower bound of the harmonic radius. We state their result as follows. 
\begin{theorem}{\cite{CheegerAnderson}} \label{thm:harmonic_radius} Let $(M,g)$ be a closed Riemannian $n$-manifold such that  for some $\xi\in\mathbb{R}$ and $i_0>0$, 
\[
\operatorname{Ric}_M\geq (n-1)\xi
\qquad\text{and}\qquad
\operatorname{inj}_M\geq i_0.
\]
Then for  $Q>1$ and $0<\alpha<1$,  there exists a positive lower bound   $r_H=r_H(Q,0,\alpha)$ of the $C^{0,\alpha}$-harmonic radius in $M$ depending only on 
$n,Q,\alpha, i_0,\xi$.\\
Additionally, if 
\[|\operatorname{Ric}_M|\leq (n-1)\xi,\]
then there exists a positive lower bound $r_H=r_H(Q,1,\alpha)$ of the $C^{1,\alpha}$-harmonic radius in $M$ depending only on 
$n,Q,\alpha, i_0,\xi$.
\end{theorem}
This uniform lower bound on the harmonic radius allows us to work in harmonic coordinate balls of a fixed radius and obtain uniform local estimates for the Hodge Laplacian. We now state our main result.
\begin{theorem} \label{thm:main_theorem}
Let $(M,g)$ be a closed Riemannian $n$-manifold satisfying 
\[\operatorname{Ric}_M\ge (n-1)\xi, \qquad \operatorname{inj}_M\ge i_0, \qquad \operatorname{Diam}_M =D.\] For $Q>1$ and $0<\alpha<1$, let $r_H$ denote the uniform lower bound of the $C^{0,\alpha}$-harmonic radius of $M$. Let $\lambda_{k,p}(M)$ denote the $k$-th eigenvalue of the Hodge Laplacian acting on $p$-forms on $M$, counted with multiplicity. Then
 
\[\lambda_{k,p}(M)\le  2{Q}^{2p}\lambda_0^D\left(B_\xi\left(\frac{D}{2k}\right)\right), \quad \forall k\ge \frac{D}{2r_H}\]
and
\[\lambda_{k,p}(M)\le  2{Q}^{2p}\lambda_0^D\left(B_\xi\left(r_H\right)\right), \quad \forall k\le \frac{D}{2r_H},\]
where $B_\xi(r)$ denotes a ball of radius $r$ in the simply connected space of constant sectional curvature $\xi$ and $\lambda_0^D(B_\xi(r))$ is the first Dirichlet eigenvalue of the Laplace-Beltrami operator on $B_\xi(r)$.
\end{theorem}
It is worth noting that obtaining eigenvalue estimates only in terms of $\xi,i_0$ and $D$ is difficult  because the explicit expression of $r_H$ in terms of these parameters is not known. For fixed \(m\) and \(\alpha\), the harmonic radius \(r_H(Q,m,\alpha)\) is a monotone increasing and continuous function of \(Q\). Moreover, from the construction in the proof of Theorem \ref{thm:harmonic_radius} by Anderson and Cheeger, one has \(r_H\le \frac{1}{2}\operatorname{inj}_M\). However, there is no known optimal choice of \(r_H\), nor is the precise relationship between \(r_H\) and the injectivity radius known. Theorem \ref{thm:harmonic_radius} guarantees existence of a uniform lower bound of the harmonic radius for manifolds with Ricci curvature and injectivity radius bounded below, but  does not provide a precise relationship between the harmonic radius and the geometric lower bounds. Thus, our eigenvalue estimates are explicit in terms of a uniform lower bound $r_H$ for the harmonic radius.
 \par  In \cite{honda}, Honda showed that for $p>1$, the spectral convergence of $\lambda_{k,p}(M)$ does not hold for a collapsed sequence of Riemannian manifolds, whereas it does hold for $\lambda_{k,1}(M)$. In view of these results, it would be interesting to replace the lower bound on the injectivity radius with a lower bound on volume in Theorem \ref{thm:main_theorem}.

 \par In a similar direction, Savo derived upper bounds for the spectral gap of the Hodge Laplacian on the hyperbolic space and established bounds for the first eigenvalue of the Hodge Laplacian on convex Euclidean domains with absolute and relative boundary conditions in \cite{savo09}. In a subsequent work \cite{guerinisavo2004}, Guerini and Savo derived an estimate of the first eigenvalue of $\Delta$ for manifolds whose boundaries have some degree of convexity.   In \cite{takahashi2001}, \cite{takahashi2003}, Takahashi derived bounds on the spectral gap between the first eigenvalues of the Laplacian on functions and on $1$-forms. Bounds on eigenvalues of the Hodge Laplacian were also obtained in \cite{Tatiana} using discretizations of the manifold assuming sectional curvature bounds.

Let $\nabla$ denote the Levi-Civita connection associated with the Riemannian metric $g$ on $M$. The connection Laplacian is defined as the operator $\nabla^*\nabla$, where $\nabla^*$ denotes the formal adjoint of $\nabla$ with respect to the $L^2$ inner product induced by $g$. Let $\lambda_{k,p}^C(M)$ denote the $k$-th positive eigenvalue of the connection Laplacian acting on smooth $p$-forms on $M$, counted with multiplicity. Using spectral convergence, Honda showed in \cite{honda} that $\lambda_{k,p}^C(M)$ is uniformly bounded in the class of manifolds with bounded Ricci curvature, diameter bounded above, and volume bounded below by some positive constant. Using harmonic coordinates, we obtain the following explicit estimates for the eigenvalues of the connection Laplacian on $p$-forms for manifolds with bounded Ricci curvature.
\begin{theorem} \label{thm:conn_theorem}
Let $(M,g)$ be a closed Riemannian $n$-manifold such that 
\[|\operatorname{Ric}_M|\le (n-1)\xi, \qquad \operatorname{inj}_M\ge i_0, \quad {\rm Diam}_M=D.\] For $Q>1$ and $0<\alpha<1$, let $r_H$  denote the lower bound of the $C^{1,\alpha}$-harmonic radius of $M$. Then 
$$
\lambda^C_{k,p}(M)
\leq
18Q^{2p}
\left(
\frac{n^2Q^3(Q-1)^2}{r_H^2}
+
\frac{2k^2}{D^2}
\right)
\frac{
V_{-\xi}\!\left(\dfrac{D}{2k}\right)
}{
V_{-\xi}\!\left(\dfrac{D}{4k}\right)
},
\qquad
\forall\, k\geq \frac{D}{2r_H}
$$

and $$
\lambda^C_{k,p}(M)
\leq
9Q^{2p}r_H^{-2}
\left( 1+
2n^2Q^3(Q-1)^2
\right)
\frac{
V_{-\xi}\!\left(r_H\right)
}{
V_{-\xi}\!\left(\dfrac{r_H}{2}\right)
},
\qquad
\forall\, k\leq \frac{D}{2r_H}.
$$

\end{theorem}

\par \textbf{Idea of the proof:}  Our approach is based on the method used by Dodziuk \cite{dodziuk1982} and Cheng \cite{Cheng}. Since the only curvature assumption imposed is a lower bound on the Ricci curvature, our methods mostly rely on the theorem by Anderson and Cheeger on the existence of a uniform harmonic radius. We used harmonic coordinate vector fields within the harmonic radius to construct test $ p$-forms to obtain an upper bound of the Dirichlet eigenvalues of the Hodge Laplacian (Lemma \ref{lem:upper bound dirichlet eigenvalue}) and the connection Laplacian (Lemma \ref{lem:connection_upper_bound_balls}) on small balls. Then we consider disjoint balls of radius smaller than the harmonic radius centered at the points of the discretization of the manifold and use bounds on harmonic coordinates in  \cite{CheegerAnderson} to derive an upper bound for the Dirichlet eigenvalues on the balls. 
\par In Theorem \ref{thm:main_theorem} and Theorem \ref{thm:conn_theorem}, we then establish an upper bound for the eigenvalues on manifolds of diameter $D$ by considering a geodesic of maximal length, partitioning it into segments of length controlled by the harmonic radius, and discretizing the manifold by extending the balls arising from this partition. We derive two immediate corollaries of Theorem \ref{thm:main_theorem} and Theorem \ref{thm:conn_theorem} according to the sign of the Ricci curvature lower bound and obtain upper bounds in terms of the volume in Theorem \ref{thm:volume} and Theorem \ref{thm:conn_volume}. 
\par On non-compact manifolds, we derive an upper bound for the spectral gap of the Hodge Laplacian by considering an exhaustion of the manifold using compact balls. We apply Theorem \ref{thm:main_theorem} to these balls to derive an upper bound of the first Dirichlet eigenvalue and take their limit to obtain an upper bound on the spectral gap.

\section{Upper Bound of the Hodge Laplacian Eigenvalues}\label{sec:upperbound}
In this section, we derive an upper bound for the eigenvalues of the Hodge Laplacian. We combine local Dirichlet eigenvalue estimates with a domain decomposition to derive a global, Cheng-type estimate.
For a domain $D\subseteq M$, we consider the Hodge Laplacian acting on 
$p$-forms with absolute boundary conditions:
\[
\iota_\nu \omega = 0, \quad \iota_\nu d\omega = 0 \text{ on } \partial D.
\]
 where $\nu$ is the outward unit normal vector field along $\partial D$. We denote the corresponding Dirichlet eigenvalues by $\lambda_{k,p}^D(D).$




\begin{defn}[$\varepsilon$-discretization] \label{defn:epsilondiscretization} For any positive $\varepsilon$, a finite set $X\subset M$ is called an \emph{$\varepsilon$--discretization} of $(M,g)$ if:
\begin{enumerate}
    \item $d(p,q)\ge 2\varepsilon$ for all distinct $p,q\in X$;
    \item the family of balls $\{B(p,2\varepsilon)\}_{p\in X}$ covers $M$.
\end{enumerate}

\end{defn}

\par We state a standard comparison principle for eigenvalues of quadratic forms under injective linear maps.
\begin{lemma}
\label{lem:comparing eigenvalues by injective map}
Let $(H_1,\langle\cdot,\cdot\rangle_1)$ and $(H_2,\langle\cdot,\cdot\rangle_2)$ be Hilbert spaces, and let
\[
\Phi : H_1 \to H_2
\]
be a linear injective map.  
Let $Q_1$ and $Q_2$ be quadratic forms on $H_1$ and $H_2$, respectively.  
Assume that there exist positive constants $C_1,C_2>0$ such that for all $f\in H_1$,
\[
\langle f,f\rangle_1 \le C_1 \langle \Phi(f),\Phi(f)\rangle_2,
\qquad
Q_1(f) \ge C_2\, Q_2(\Phi(f)).
\]
Then, for every $k\in\mathbb{N}$,
\[
\lambda_k(Q_1) \ge \frac{C_2}{C_1}\,\lambda_k(Q_2).
\]
\end{lemma}

The next lemma provides a domain decomposition estimate for the Hodge Laplacian. This allows us to control global eigenvalues in terms of local Dirichlet problems.

\begin{lemma}
\label{lem: manifoldanddirichlet}    
Let $(M,g)$ be a closed Riemannian $n$-manifold and $D_1,\dots,D_j \subset M$ be pairwise disjoint domains in M. Let $\lambda^{D}_{j,p}(D_i)$ be the eigenvalue of the Hodge Laplacian $\Delta$ acting on $p$-forms on $D_i$
with Dirichlet boundary conditions. Then, 
\begin{equation*}
\lambda_{jl - 1,p}(M)
\;\le\;
\max_{1 \le i \le j} \lambda^{D}_{l-1,p}(D_i). 
\end{equation*}
\end{lemma}

\begin{proof}

Define the Hilbert space $H_1 = \bigoplus_{i=1}^{j} H^1_0\!\left(\Lambda^p T^*D_i\right)$ and for $\omega=\left(\omega_1,\omega_2,\dots,\omega_j\right)$ let
\[
\langle \omega,\omega \rangle_{H_1}
=
\sum_{i=1}^{j} \int_{D_i} |\omega_i|^2 \, dv_g,
\]
and quadratic form
\[
Q_1(\omega)
=
\sum_{i=1}^{j}
\int_{D_i}
\bigl(
|d\omega_i|^2 + |d^*\omega_i|^2
\bigr)\, dv_g.
\]

Each $\omega_i$ is extended by zero outside $D_i$, which is admissible since the domains are pairwise disjoint and the boundary conditions are Dirichlet.
From Lemma \ref{lem:comparing eigenvalues by injective map},
\[
\frac{Q_1(\omega)}{\langle \omega,\omega \rangle_{H_1}}
\;\le\;
\max_{1 \le i \le j} \lambda^{D}_{l-1,p}(D_i).
\]
Hence the min-max principle for $\Delta$ on $M$ yields the result.
\end{proof}

As a consequence of the domain decomposition estimate, we obtain an upper bound on the global eigenvalues in terms of local Dirichlet problems on the discretization balls.

\begin{corollary}
    \label{cor:upper-bound-max}
Let $\varepsilon<r_H$, $X=\{x_i\}$ be an $\varepsilon$--discretization of $M$ and $\{B_i=B(x_i,\varepsilon)\}$.  Then,
\[\lambda_{|X|-1,p}(M)\leq\max_i\lambda_{0,p}^{D}(B_i).\]
\end{corollary}


The following lemma establishes a uniform upper bound for the Dirichlet eigenvalues on geodesic balls of radius $\varepsilon$ smaller than the harmonic radius $r_H$ for $n$-dimensional manifolds  $(M,g)$ with Ricci curvature bounded below. 
This estimate will later be used to derive upper bounds for $\lambda_{k,p}$ on balls centered at the points of an $\varepsilon$-discretization of $M$.
\begin{lemma}    \label{lem:upper bound dirichlet eigenvalue}
Let $(M,g)$ be a closed Riemannian  $n$-manifold with \[\operatorname{Ric}_M \ge \xi, \qquad \operatorname{inj}_M\ge i_0.\] For $Q>1$ and $0<\alpha<1$, let $r_H$ denote the lower bound on the $C^{0,\alpha}$-harmonic radius of $M$. For $0<\varepsilon\le r_H$ and $x\in M$, let $B=B(x, \varepsilon)$ be a geodesic ball in $M$ centered at $x$ with radius $\varepsilon$ and $B_\xi(\varepsilon)$ be a ball of radius $\varepsilon$ in the space of constant curvature $\xi.$ Then   
\[\lambda_{0,p}^D(B)\le 2Q^{2p}\lambda_0^D(B_\xi(\varepsilon))\]
where $\lambda_0^D(B)$ denotes the first Dirichlet eigenvalue of the Laplace-Beltrami operator.
\end{lemma}
\begin{proof}
Since $\varepsilon<r_H$, we choose  harmonic coordinates $\{y_j\}_{j=1}^n$ in $B$ and  define
\[
\omega_0 = dy^{1}\wedge\cdots\wedge dy^{p}.
\]
For a $p$-form $\alpha=\sum_{i_1<i_2<\cdots <i_{p}}\alpha_{i_1i_2\cdots  i_{p}}dx^{i_1}\wedge dx^{i_2}\cdots dx^{i_{p}}$, from (3.3.47) \cite{jost}
\begin{equation}
\label{codifferential_expansion}
    \left(d^*\alpha\right)_{i_1i_2\cdots i_{p-1}}=-g^{kl}\left(\frac{\partial \alpha_{ki_1i_2\cdots i_{p-1}}}{\partial x^l}-\Gamma^j_{kl}\alpha_{ji_1\cdots i_{p-1}}\right).
\end{equation}
For harmonic coordinates $\{y_j\}_{j=1}^n$ in $B$,  we have $g^{kl}\Gamma_{kl}^j=0$ for all $j=1,2,\cdots,n$. Hence, from (\ref{codifferential_expansion}) 
\begin{equation}
    \label{codifferential of omega0}
    d^*\omega_0=0.
\end{equation}

Let $f\in H^1_0(B)$ and $\omega=f\omega_0$. By the min-max principle,
\begin{equation}
    \label{eq:rayleigh_quotient_dirichlet}
    \lambda_{0,p}^{D}(B)\le \frac{\int_{B} (|d\omega|^2+|d^*\omega|^2)dv_g}{\int_{B}|\omega|^2 dv_g}.
\end{equation}

Now, \begin{equation} \label{eq:domegabound}
 |d\omega|^2=|df\wedge\omega_0|^2\le |df|^2|\omega_0|^2
\end{equation}
On the other hand, using (\ref{codifferential of omega0}), 
\begin{equation}
\label{eq:codifferential_in_coords}
    d^*f\omega_0= \left(g^{1l}\frac{\partial f}{dy^l}\right)dy^2\wedge \cdots\wedge dy^p.
\end{equation}
Using diagonalization and change of variables in Theorem \ref{thm:harmonic_radius} we can derive for all $1\le i,j\le n$,
\begin{equation}
 \frac{1}{Q}\delta_{ij}\le g^{ij}\le Q\delta_{ij} \nonumber
\end{equation}
where $(g^{ij})=(g_{ij})^{-1}.$ These also give bounds on the eigenvalues of all submatrices of $(g^{ij})$ and hence, their determinants. Also $|dx^1(\triangledown f)|^2\le |df|^2 .$
Using these bounds in (\ref{eq:codifferential_in_coords}), we have
\begin{equation}
    \label{eq:norm_codifferential_bound}
    |d^*f\omega_0|^2\le Q^{p}|df|^2.
\end{equation} 
Also 
\begin{equation}
  \label{eq:bound_on_norm_omega0}  
  |\omega_0|^2\ge \frac{1}{Q^p}.
\end{equation}

Substituting (\ref{eq:domegabound}), (\ref{eq:norm_codifferential_bound}) and (\ref{eq:bound_on_norm_omega0}) in (\ref{eq:rayleigh_quotient_dirichlet}),
\begin{equation}
\label{eq:test_form}\lambda_{0,p}^D(B)\le 2Q^{2p} \frac{\int_B |df|^2v_g}{\int_B|f|^2v_g}
\end{equation}
for all $f\in H^1_0(B)$. Hence, 
\[\lambda_{0,p}^D(B)\le 2Q^{2p} \lambda_{0}^D(B).\] 

Hence, from Theorem 1.1 in \cite{Cheng},
\[\lambda_{0,p}^D(B)\le 2Q^{2p}\lambda_0^D(B_\xi(\varepsilon)).\qedhere\]

\end{proof}

In particular, we obtain a similar result related to Lott's eigenvalue bounds \cite{LottJ} where we use Ricci curvature and injectivity radius bounds instead of sectional curvature and volume constraints.
\subsection{Proof of Theorem \ref{thm:main_theorem}: }

    Since $(M,g)$ is closed, by the Hopf-Rinow theorem, there exists a minimizing geodesic $\gamma$ of length $D$. We choose the following decomposition of $\gamma$:
    \[\gamma(0),\gamma\left(\frac{D}{k}\right),\gamma\left(\frac{2D}{k}\right),\cdots,\gamma\left(\frac{(k-1)D}{k}\right),\gamma(D).\]
Since $\gamma$ is minimizing, each of the points in the above partition are equally spaced by distance $\frac{D}{k}$. Let $B_i=B\left(\gamma\left(\frac{iD}{k}\right),\frac{D}{2k}\right)$. For $i\neq j$, if $p\in B_i\cap B_j$, then \[d\left(\gamma\left(\frac{iD}{k}\right),\left(\frac{jD}{k}\right)\right)\le d\left(p,\left(\frac{iD}{k}\right)\right)+ d\left(p,\left(\frac{jD}{k}\right)\right)<\frac{D}{k}\] which is a contradiction. Hence, $B_i\cap B_j=\phi$ for all $i\neq j.$ \\ Choose $\varepsilon= \frac{D}{2k}$ and let $X(\varepsilon)$ be an $\varepsilon$-discretization of $M$ by choosing an extension of the balls $B_i$ defined above. Clearly, $|X(\varepsilon)|\ge k+1$. For $k\ge\frac{D}{2r_H}$, we have $\varepsilon\le r_H$. 

Hence, we can apply Lemma \ref{lem:upper bound dirichlet eigenvalue}  and conclude that 
\[\lambda_{k,p}(M)\le  2Q^{2p}\lambda_0^D\left(B_\xi\left(\frac{D}{2k}\right)\right), \quad \forall k\ge \frac{D}{2r_H}\]
and \[\lambda_{k,p}(M)\le 2Q^{2p} \lambda_0^D\left(B_\xi(r_H\right)),\qquad \forall k\le \frac{D}{2r_H}.\qedhere\]

As a consequence of the above theorem, we obtain the following corollaries.

\begin{corollary}
\label{cor:positive_ricci}
Let $(M,g)$ be a closed Riemannian manifold with non-negative Ricci curvature. For $Q>1$ and $0<\alpha<1$, if the $C^{0,\alpha}$-harmonic radius is bounded below by $r_H$ and $\mathrm{ Diam}_M= D$ then 

\[\lambda_{k,p}(M)\le 2Q^{2p}n^2\pi^2\frac{k^2}{D^2},\qquad \forall k\ge\frac{D}{2r_H}\]
and \[\lambda_{k,p}(M)\le 2Q^{2p}n^2\pi^2r_H^{-2},\qquad \forall k\le \frac{D}{2r_H}.\]

\end{corollary}

\begin{proof}
    From the proof of Theorem 2.1 in \cite{Cheng}, we observe that there is an n-cube with sides $\frac{2}{\sqrt{n}}\frac{D}{2k}$ contained in $B_n(\xi,\frac{D}{2k})$. Using the first eigenvalue of the cube and the domain monotonicity theorem, the result follows.
\end{proof}
For a non-positive lower bound on the Ricci curvature, we have the following upper bound on $\lambda_{k,p}(M).$ 
\begin{corollary} Let $(M,g)$ be a closed Riemannian manifold with $\rm{Ric}_M\geq -(n-1)\xi$, $\xi>0$. For $Q>1$ and $0<\alpha<1$, if the $C^{0,\alpha}$-harmonic radius is bounded below by $r_H$ and $\mathrm{Diam}_M= D$ then 
\begin{enumerate}
    \item For $n=2(m+1)$ where $m=0,1,2,\cdots$

\[\lambda_{k,p}(M)\le 2Q^{2p}(2m+1)^2\xi+2Q^{2p+2}(1+2^m)^2\pi^2\frac{k^2}{D^2},\qquad \forall k\ge\frac{D}{2r_H}\]
and \[\lambda_{k,p}(M)\le 2Q^{2p}(2m+1)^2\xi+2Q^{2p+2}(1+2^m)^2\pi^2r_H^{-2},\qquad \forall k\le \frac{D}{2r_H}.\]
\item For $n=2m+3$, $m=0,1,2,\cdots$
                    \[\lambda_{k,p}(M)\le 2Q^{2p}(2m+2)^2\xi+2Q^{2p+2}(1+2^{2m})^2(1+\pi^2)\frac{k^2}{D^2},\qquad \forall k\ge\frac{D}{2r_H}\]
     and \[\lambda_{k,p}(M)\le 2Q^{2p}(2m+2)^2\xi+2Q^{2p+2}(1+2^{2m})^2(1+\pi^2)r_H^{-2},\qquad \forall k\le \frac{D}{2r_H}.\]       
\end{enumerate}  
\end{corollary}
\begin{proof}
  Using Corollary 2.3 \cite{Cheng} in Lemma \ref{lem:upper bound dirichlet eigenvalue}, we get the required bounds.
\end{proof}

\begin{theorem}
\label{thm:volume}
 Let  $(M,g)$ be a closed Riemannian $n$-manifold with volume $V$. Assume that the Ricci curvature satisfies $\mathrm{Ric}_M \ge -(n-1)$ and the injectivity radius is bounded below by $i_0$. Then   for $n\ge 3$, $Q>1$ and $0<\alpha<1$, if the $C^{0,\alpha}$-harmonic radius is bounded below by $r_H$, we have

\[ \lambda_{k,p}(M)\le 2Q^{2p+\frac{3n}{2}-1} \left(\frac{n\alpha_{n-1}}{V}\right)\left(1+\frac{V}{\alpha_{n-1}(k+1)}\right)^3  (k+1),\qquad \mathrm{for }\hspace{1mm} k\ge \frac{2^nV}{\alpha_{n-1}r_H^n}.
\]
where $\alpha_n=\mathrm{vol}(S^n).$
\end{theorem}
\begin{proof}
We choose \[\varepsilon
=
2\operatorname{arcsinh}\left(\frac{ V}{\alpha_{n-1}(k+1)}\right)^{1/n}.
\]

From Bishop's theorem, 
\begin{equation}
\label{eq:volume_ratio}
  |X(\varepsilon)|\ge \frac{V}{\sup_{y\in X(\varepsilon)}\{B(y,\varepsilon/2)\}}\ge \frac{V}{V_{-1}(\varepsilon/2)}.  
\end{equation}
Let $V_{-1}(\varepsilon)$ be the volume of a ball of radius $\varepsilon$ in the space of constant curvature $-1$. Then
\begin{align}
V_{-1}(\varepsilon/2)
&=
\alpha_{n-1}
\int_{0}^{\varepsilon/2}
\left(
\sinh t
\right)^{n-1}
dt \nonumber\\
&\le
\alpha_{n-1}
\left(
\sinh\left(\frac{\varepsilon}{2}\right)
\right)^n \nonumber\\ &=\frac{V}{k+1}.\nonumber
\end{align}

Substituting the above bound in (\ref{eq:volume_ratio}),

\begin{align}
|X(\varepsilon)| \ge k+1.
\end{align}
Now $y\le \sinh{y}\implies \frac{1}{y^n}\ge \frac{1}{(\sinh y)^n}$. Hence for $k>\frac{2^nV}{\alpha_{n-1}r_H^n}$, we have $k>\frac{V}{\alpha_{n-1}(\sinh(r_H/2))^n}.$ Hence,  $\varepsilon<r_H$ for $k>\frac{2^nV}{\alpha_{n-1}r_H^n}$ . Using Lemma \ref{lem:upper bound dirichlet eigenvalue}, for $q\in M$ let $f:B(q,\varepsilon)\rightarrow\mathbb{R}$ be defined as $f(y)=1-\frac{1}{\varepsilon}d(q,y)$. Then $|df|\ge \frac{1}{\varepsilon}$ and $|f|\ge\frac{1}{2}$ in $B\left(q,\frac{\varepsilon}{2}\right)$. Hence, using $\omega=f\omega_0$ as the test form in equation \ref{eq:test_form} and using Bishop-Gromov volume ratio, we have
\begin{equation}
    \label{eq:lambda upper bound for volume}
    \lambda_{k,p}(M)\le 2Q^{2p}\frac{V_{-1}(\varepsilon)}{\varepsilon^2 V_{-1}(\frac{\varepsilon}{2})}, \qquad \mathrm{for}\hspace{1mm} k>\frac{2^nV}{\alpha_{n-1}r_H^n}.
\end{equation}
Let $x^n=\frac{ V}{\alpha_{n-1}(k+1)}$. Then \begin{align}
    \frac{V_{-1}(\varepsilon)}{ V_{
-1}(\frac{\varepsilon}{2})} &\le \left(\frac{\sinh\varepsilon}{\sinh\left(\varepsilon/2\right)}\right)^n \nonumber\\
    &\le 2^n\left(\cosh\left(\varepsilon/2\right)
    \right)^n \nonumber \\
    &\le 2^{\frac{3n}{2}-1}\left(1+x^n\right)\label{eq:cosh}
\end{align}
where \ref{eq:cosh} follows from using Jensen's inequality as follows \[\left(\cosh\left(\frac{\varepsilon}{2}\right)\right)^n\le \left(1+\sinh^2\left(\frac{\varepsilon}{2}\right)\right)^\frac{n}{2}=\left(1+x^2\right)^\frac{n}{2}\le 2^{\frac{n}{2}-1}(1+x^n).\] 
For $x=\left(\frac{V}{\alpha_{n-1}(k+1)}\right)^\frac{1}{n}$, $\sinh^{-1}x\ge \frac{1}{n}\left(\frac{x^\frac{n}{2}}{1+x^n}\right)$. This is because the derivative of $\sinh^{-1}x- \frac{1}{n}\left(\frac{x^\frac{n}{2}}{1+x^n}\right)$ is greater than 0 for all non-negative values. Substituting $\varepsilon=2\sinh^{-1}x$ in (\ref{eq:lambda upper bound for volume})  gives
\begin{equation}
\label{eq:upper bound after sinh substitution}
    \lambda_{k,p}(M)\le 2Q^{2p+\frac{3n}{2}-1}\cdot n\cdot\frac{(1+x^n)^3}{x^n}.
\end{equation}

Substituting the value of $x$ and $Q$ in (\ref{eq:upper bound after sinh substitution}),
\[ \lambda_{k,p}(M)\le 2Q^{2p+\frac{3n}{2}-1} \left(\frac{n\alpha_{n-1}}{V}\right)\left(1+\frac{V}{\alpha_{n-1}(k+1)}\right)^3 (k+1).\]

\end{proof}

On non-compact complete manifolds, since the spectrum may contain continuous components, we derive bounds on the spectral gap of the Hodge Laplacian acting on $p$-forms. This can be done by obtaining an upper bound on the supremum of the $L^2$-spectrum defined as follows. 
\begin{defn}
    Let M be a $n$-dimensional Riemannian manifold without boundary. The  supremum of $L^2$ spectrum of $p$-forms is defined as
    \[\sigma^p(M)=\inf_{\omega\in \Lambda^p_0(M)}\frac{\int\|d\omega\|^2+\|d^*\omega\|^2}{\int \|\omega\|^2}.\]
    
\end{defn}

Savo \cite{savo09}  obtained the following bound on the supremum of $L^2$ spectrum on $p$-forms on the hyperbolic space $\mathbb{H}^n$.
\begin{equation*}
\sigma^p(\mathbb{H}^n)=
\begin{cases}
0, & \text{if $p\le (n+1)/2$, }\\
\frac{(2p-n-1)^2}{4}, & \text{if $p>(n+1)/2$}.
\end{cases}
\end{equation*}
Savo also showed that if  $M$ admits a non-trivial harmonic $p$-form then $\sigma_p(M)=0$ \cite{savo09}. As a corollary of Theorem \ref{thm:main_theorem}, we can derive an upper bound for the supremum of the $L^2$ spectrum on complete non-compact manifolds whose Ricci curvature admits a uniform lower bound.
 
\begin{corollary}

   Let $(M,g)$ be a complete non-compact Riemannian $n$-manifold such that $\mathrm{Ric}_M\ge (n-1)\xi.$ Then for $Q>1$ and $0<\alpha<1$ there exists a uniform lower bound $r_H$ of the $C^{0,\alpha}$-harmonic radius such that for $p$-forms
   \[\sigma^p(M)\le 2Q^{2p} \lambda_0^D(B_\xi(r_H)).\]   
   \end{corollary}
  \begin{proof}

 Let $x\in M$ and $\{B_i=\overline{B\left(x,i\right)}\}_{i\in\mathbb{N}}$ be an exhaustion of $M$ by compact domains. Fixing $Q>1,0<\alpha<1$ and a lower bound $r_0$ of the injectivity radius, from \cite{Hebey}, the $C^{0,\alpha}$-harmonic radius is bounded below by a constant $r_H$.  We choose $j\in\mathbb{N}$ such that  $2r_H\le j.$      Then for all $j'>j$, from Theorem \ref{thm:main_theorem}, 
\[\lambda_{0,p}^D(B_j')\le 2Q^{2p} \lambda_0^D(B_\xi(r_H)).\]      
     Since $\Lambda^p_0(B_i)\subseteq \Lambda^p_0(B_{i+1})$ we have $\sigma^p(M)\le \lim_{i\rightarrow \infty}\lambda_{0,p}^D(B_i)$. Hence, \[\sigma^p(M)\le 2Q^{2p} \lambda_0^D(B_\xi(r_H)).\]

  \end{proof}

We also have the following corollary of Theorem \ref{thm:main_theorem} for the first eigenvalue of the connection Laplacian on $1$-forms.

\begin{corollary} \label{cor:connectionLap}
    Let $(M,g)$ be a closed $n$-manifold with ${\rm inj}_M \ge r_0$, ${\rm Diam}_M \le D <\infty$, $ {\rm Ric}_M \ge 0$. For $Q>1$ and $0<\alpha<1$, let $r_H$ denote the uniform lower bound on the $C^{0,\alpha}$-harmonic radius of M. Then the first non-zero eigenvalue $\lambda^C_{1,1}(M)$ of the connection Laplacian $\nabla^{*} \nabla$ acting on the space of $1$-forms satisfies
   \[ \lambda_{1,1}^C(M)\le \frac{2Q^2n^2\pi^2}{r_H^2}.\]
\end{corollary}
\begin{proof} For a smooth Riemannian manifold $(M,g)$ and a smooth $1$-form $\omega$, we recall the Bochner (Weitzenb\"ock) formula for $1$-forms:
\begin{equation}
\label{eq: weitzenbock pointwise}
\Delta \omega
=
\nabla^{*}\nabla \omega
+
\mathrm{Ric}(\omega),
\end{equation}
where $\nabla^{*}\nabla$ is the rough (connection) Laplacian, and $\mathrm{Ric}(\omega)$ is the action of the Ricci tensor on $\omega$.

Taking the $L^2$-inner product of \eqref{eq: weitzenbock pointwise} with $\omega$  and integrating over $M$, we obtain
\[
\int_M \langle \Delta \omega , \omega \rangle
=
\int_M \langle \nabla^{*}\nabla \omega , \omega \rangle
+
\int_M \mathrm{Ric}(\omega,\omega).
\]

Since $M$ is closed, by integration by parts we obtain
\begin{equation*}
\int_M \bigl( |d \omega|^2 +  |\delta \omega|^2\bigr)
=
\int_M |\nabla \omega|^2
+
\int_M \mathrm{Ric}(\omega,\omega).
\end{equation*}

For $\mathrm{Ric}\ge 0$, on $1$-forms we have $\mathrm{ker}(\nabla^*\nabla)=\mathrm{ker}(\Delta)$. Also,
\begin{equation*}
   \frac{\int_M |\nabla \omega|^2}{\int_M |\omega|^2}\le \frac{\int_M \bigl( |d \omega|^2 +  |\delta \omega|^2\bigr)}{\int_M | \omega|^2 }.
\end{equation*}
Applying min-max principle,
\[\lambda_{1,1}^C(M)\le \lambda_{1,1}(M).\]
We then use Corollary \ref{cor:positive_ricci} and domain monotonicity of eigenvalues to obtain the desired inequality.
\end{proof}
\section{Upper Bound of the Connection Laplacian Eigenvalues}

In this section, we obtain upper bounds for the eigenvalues of the connection Laplacian on \(p\)-forms. The proof requires control of the first derivatives of the metric coefficients in harmonic coordinates. Under a lower Ricci curvature bound, such control is available through the \(H^q_1\)-harmonic radius, whereas under a two-sided Ricci curvature bound one has the stronger \(C^{1,\alpha}\)-control of the metric coefficients. We first introduce the \(H^q_1\)-harmonic radius and use it to obtain the eigenvalue estimates under a lower Ricci curvature bound. We then discuss the corresponding estimates under a two-sided Ricci curvature bound.


\begin{defn}
    Let $(M, g )$ be a closed Riemannian $n$-manifold and let $x\in M$ . Given $Q > 1$, $ m \in \mathbb{N}$, and $q> n$, the $H^q_m$-harmonic radius at $x$ is the largest number $r_H = r_H (Q, m, q)(x)$ such that, on the geodesic ball $B(x,r_H)$, there is a harmonic coordinate chart such that\begin{enumerate}
        \item $Q^{-1}\delta_{ij}\leq g_{ij}\leq Q\delta_{ij}$ as bilinear forms, and
        \item $\sum_{1\leq |\beta|\leq m}
    r_H^{|\beta|-n/q}
    \left\|\partial^\beta g_{ij}\right\|_{L^q(B(x,r_H))}
    \leq Q-1.$
    \end{enumerate}
\end{defn}

\begin{theorem}\cite{Hebey}
Let $(M,g)$ be a closed Riemannian $n$-manifold,
and let $Q>1$ and $q>n$.
 Suppose that there exist constants $\xi\in\mathbb{R}$
    and $i_0>0$ such that
    \[
        \operatorname{Ric}_M\geq (n-1)\xi 
        \qquad\text{and}\qquad
        \operatorname{inj}_M\geq i_0.
    \]
    Then there exists a uniform lower bound $r_H=r_H(Q,1, q)$ of the $H^q_1$-harmonic radius depending only on $n,Q,q,i_0,\xi$.  \end{theorem}

\begin{lemma}
\label{lem:connection_upper_bound_balls}
    Let $(M,g)$ be a closed Riemannian $n$-manifold such that
    \[
        \operatorname{Ric}_M\geq (n-1)\xi 
        \qquad\text{and}\qquad
        \operatorname{inj}_M\geq i_0.
    \]
     For $Q>1$, let $r_H$ be the lower bound of the $H^{2n}_1$-harmonic radius of $M$. For $x\in M$, $0<\varepsilon<r_H$, let $B=B(x,\varepsilon)$.  Let $\lambda_{0,p}^C(B)$ denote the first non-zero Dirichlet eigenvalue of the connection Laplacian on $p$-forms in $B$. Then
    \[\lambda_{0,p}^C(B)\le 9Q^{2p}\left(1+\frac{n^{n+2}Q^{p+3}(Q-1)^2\alpha_n^{n-1}}{2^{n-2}r_H\alpha_{n-1}^n}\right)\frac{V_\xi(\varepsilon)}{\varepsilon^2 V_\xi(\frac{\varepsilon}{2})}\]
    where $ \alpha_n=\mathrm{vol}(S^n)$ and $B_\xi(\varepsilon)$ be a ball of radius $\varepsilon$ in the space of constant curvature $\xi$.\\
    
\end{lemma}
\begin{proof}
    Since $\varepsilon<r_H$, we choose   harmonic coordinates $\{y_j\}_{j=1}^n$ in $B$ and  define
\[
\omega_0 = dy^{1}\wedge\cdots\wedge dy^{p}.
\]
On $B$ we define the cut-off function
\[f:B\rightarrow \mathbb{R} \quad \text{by} \quad f(y)=1-\frac 1\varepsilon d(x,y).\]
 Let $\omega=f\omega_0$. By the minimax principle,
\[\lambda_{0,p}^{C}(B)\le \frac{\int_{B} |\nabla \omega|^2 dv_g}{\int_{B}|\omega|^2 dv_g}.\]
Using Leibniz rule for connections, 
\[
\nabla(f\omega_0) = df \otimes \omega_0 + f \nabla \omega_0.
\]
Using Cauchy-Schwarz inequality and bounds on the cutoff function,
\begin{align}
  |\nabla(f\omega_0)|^2 &\le 2|df \otimes \omega_0|^2 + 2|f \nabla \omega_0|^2  \nonumber\\
  &\le 2|df|^2|\omega_0|^2+2|\nabla\omega_0|^2. \label{eq:bound_on_connection_test_form}
\end{align}
Now 
\begin{align}
    |\nabla \omega_0|^2&=g^{ij}\langle \nabla_{\partial_i} \omega_0, \nabla_{\partial_j} \omega_0\rangle \nonumber\\
    &=g^{ij}\langle \Gamma_{il}^l \omega_0, \Gamma_{jk}^k \omega_0\rangle \nonumber\\
   & =g^{ij}\Gamma_{il}^l \Gamma_{jk}^k|\omega_0|^2. \label{eqndstarw}
\end{align}
For $x\in M$, let $|\partial g(x)|:=\max_{i,j,k}\{|\partial_kg_{ij}(x)|\}$ and $|\Gamma(x)|:=\max_{i,j,k}\{|\Gamma^k_{ij}(x)|\}$. The metric bounds in Definition \ref{thm:harmonic_radius}  also imply uniform control of the eigenvalues of $g^{-1}$.  As 
\[\Gamma^k_{ij}=\frac12g^{k\ell}\bigl(\partial_i g_{j\ell}+\partial_j g_{i\ell}-\partial_\ell g_{ij}\bigr),
\]
we have,
\[|\Gamma(x)|\le \frac{3nQ}{2}\,|\partial g(x)|.
\]
Since $|\omega_0(x)|^2\le Q^p$, from  equation (\ref{eqndstarw}) we have,
\[|\nabla\omega_0(x)|^2\le  \frac{9n^2Q^{p+3}}{4}|\partial g(x)|^2. \]
Using Theorem \ref{thm:harmonic_radius} again we have, $ \partial g\in L^q(B)\subseteq L^2(B).$
Using H\"{o}lder's inequality for $q=2n$, we have
\[\|\partial g\|_{L^2(B)}^2\le \operatorname{Vol}(B)^{\frac{n-1}{n}}\|\partial g\|_{L^{2n}(B)}^2\le \operatorname{Vol}(B)^{\frac{n-1}{n}}\frac{(Q-1)^2}{r_H}. \]
Hence, 
\[\int_{B}|\nabla\omega_0|^2v_g\le \frac{9n^2Q^{p+3}(Q-1)^2}{4r_H}\operatorname{Vol}(B)^{\frac{n-1}{n}}.\]
Using $f(x)\geq \frac12$ for all $y\in B(x,\frac\varepsilon2)$, we have,
\begin{equation}
    \label{eq:bound_on_d}
\frac{\int_{B}|\nabla\omega_0|^2 dv_g}{\int_{B} |\omega|^2 dv_g}\le \frac{9n^2Q^{2p+3}(Q-1)^2}{r_H\operatorname{Vol}^{\frac{1}{n}}(B(x,
\varepsilon))}\frac{\operatorname{Vol}(B(x,\varepsilon))}{\operatorname{Vol}(B(x,\varepsilon/2))}\le \frac{9n^2Q^{2p+3}(Q-1)^2}{r_Hc(n)\varepsilon}\frac{V_\xi(\varepsilon)}{ V_\xi(\frac{\varepsilon}{2})}. 
\end{equation}
The last inequality follows from the Bishop-Gromov relative volume comparison theorem and the fact that $\operatorname{Vol}(B(x,\varepsilon))\ge c(n)\varepsilon^n$ for $c_n=\frac{2^{n-1}\alpha_{n-1}^n}{n^n\alpha_n^{n-1}}$ and, $\varepsilon\le r_H$, which is a consequence of Theorem \ref{thm:harmonic_radius}. Since $|df|=\frac1\varepsilon$ a.e. and $f(y)\geq \frac12$ for all $y\in B(x,\frac\varepsilon2)$ in equation \ref{eq:bound_on_connection_test_form}, using the Bishop-Gromov volume comparison theorem again we obtain from equations \ref{eq:bound_on_connection_test_form} and \ref{eq:bound_on_d},
\begin{equation*}
\label{eq:bound on omega}
   \frac{\int_{B}|\nabla\omega|^2 dv_g}{\int_{B} |\omega|^2 dv_g}\le  9Q^{2p}\left(1+\frac{2n^2Q^{p+3}(Q-1)^2}{c(n)r_H}\right)\frac{V_\xi(\varepsilon)}{\varepsilon^2 V_\xi(\frac{\varepsilon}{2})}. \qedhere
\end{equation*}

\end{proof}
\begin{theorem} 
\label{thm:conn_ric_lower}
Let $(M,g)$ be a closed Riemannian $n(\ge 3)$-manifold such that 
\[\operatorname{Ric}_M\ge (n-1)\xi, \qquad \operatorname{inj}_M\ge i_0.\] For $Q>1$ , let $r_H$ be the lower bound of the $H_1^{2n}$-harmonic radius of $M$. Then 
\[\lambda^C_{k,p}(M)\le 36Q^{2p}\left(1+\frac{n^{n+2}Q^{p+3}(Q-1)^2\alpha_n^{n-1}}{2^{n-2}r_H\alpha_{n-1}^n}\right)\frac{V_\xi(\frac{D}{2k})}{ V_\xi(\frac{D}{4k})}\frac{k^2}{D^2},\qquad \forall k\ge\frac{D}{2r_H}.\]
and \[\lambda^C_{k,p}(M)\le9Q^{2p}\left(1+\frac{n^{n+2}Q^{p+3}(Q-1)^2\alpha_n^{n-1}}{2^{n-2}r_H\alpha_{n-1}^n}\right)\frac{V_\xi(r_H)}{r_H^2 V_\xi(\frac{r_H}{2})},\qquad \forall k\le\frac{D}{2r_H}.\]
\end{theorem}
\begin{proof}
    
By choosing the quadratic form in Lemma \ref{lem: manifoldanddirichlet} as $Q_1(\omega)=\sum_{i=1}^j\int_{D_i}|\nabla\omega|^2dv_g$
, analogous to Corollary \ref{cor:upper-bound-max} we have $\lambda_{|X|-1,p}^C(M)\le \max_i \lambda_{0,p}^C(B_i).$ We use Lemma \ref{lem:connection_upper_bound_balls} and follow the arguments of proof of Theorem \ref{thm:main_theorem} to obtain an upper bound of the $k$-th eigenvalue of the connection Laplacian on $p$-forms.
\end{proof}

As a consequence of  Theorem \ref{thm:conn_ric_lower}, we obtain the following corollaries. Under the volume doubling condition for manifolds with non-negative Ricci curvature, the upper bound for the eigenvalues simplifies to the following.

\begin{corollary} Let $(M,g)$ be a closed Riemannian manifold with non-negative Ricci curvature.  For $Q>1$, if the $H^{2n}_1$-harmonic radius is bounded below by $r_H$ and $\mathrm{Diam}_M= D$ then 

\[\lambda^C_{k,p}(M)\le36Q^{2p}\left(1+\frac{n^{n+2}Q^{p+3}(Q-1)^2\alpha_n^{n-1}}{2^{n-2}r_H\alpha_{n-1}^n}\right)\frac{k^2}{D^2},\qquad \forall k\ge\frac{D}{2r_H}\]
and \[\lambda^C_{k,p}(M)\le 9Q^{2p}r_H^{-2}\left(1+\frac{n^{n+2}Q^{p+3}(Q-1)^2\alpha_n^{n-1}}{2^{n-2}r_H\alpha_{n-1}^n}\right),\qquad \forall k\le \frac{D}{2r_H}.\]

\end{corollary} 
For a non-positive lower bound on the Ricci curvature, we have the following upper bound on $\lambda_{k,p}(M).$ 
\begin{corollary} Let $(M,g)$ be a closed Riemannian manifold with $\mathrm{Ric}_M\geq -(n-1)\xi$, $\xi\geq 0$.  For $Q>1$, if the $H^{2n}_1$-harmonic radius is bounded below by $r_H$ and $\mathrm{Diam}_M= D$ then 
\[\lambda^C_{k,p}(M)\le 36Q^{2p}\left(1+\frac{n^{n+2}Q^{p+3}(Q-1)^2\alpha_n^{n-1}}{2^{n-2}r_H\alpha_{n-1}^n}\right)\left(\cosh\left(\frac{D\sqrt{\xi}}{4k} \right)\right)^n\frac{k^2}{D^2},\qquad \forall k\ge\frac{D}{2r_H}\]
and \[\lambda^C_{k,p}(M)\le 9Q^{2p}r_H^{-2}\left(1+\frac{n^{n+2}Q^{p+3}(Q-1)^2\alpha_n^{n-1}}{2^{n-2}r_H\alpha_{n-1}^n}\right)\left(\cosh\left(r_H \sqrt{\xi}/2\right)\right)^n,\qquad \forall k\le \frac{D}{2r_H}.\]
  
\end{corollary}
\begin{proof}
    For $\xi<0$,  we have\[\frac{V_\xi(\varepsilon)}{V_\xi(\frac{\varepsilon}{2})}\le \left[\frac{\sinh\left(\varepsilon\sqrt{\xi}\right)}{\sinh\left(\varepsilon\sqrt{\xi}/2\right)}\right]^n\]
which gives the required bound. 
\end{proof}

We also obtain an upper bound of the eigenvalues of the connection Laplacian in terms of the volume of the manifold in the following theorem.
\begin{theorem}
\label{thm:conn_volume}
 Let  $(M,g)$ be a closed Riemannian $n$-manifold with volume $V$. Assume that the Ricci curvature satisfies $\mathrm{Ric}_g \ge -(n-1)$ and for $Q>1$, the $H^{2n}_1$-harmonic radius is bounded below by $r_H>0$. Then  for $n\ge 3$ and $\alpha_n=\mathrm{vol}(S^n)$,

\[
\lambda^C_{k,p}(M)
\le
36Q^{2p}\left(1+\frac{n^{n+2}Q^{p+3}(Q-1)^2\alpha_n^{n-1}}{2^{n-2}r_H\alpha_{n-1}^n}\right)\left(1+\frac{V}{\alpha_{n-1}(k+1)}\right)^3 \frac{(k+1)}{V},\qquad \mathrm{for }\hspace{1mm} k>\frac{2^nV}{\alpha_{n-1}r_H^n}.
\]

\end{theorem}
\begin{proof}
   The proof follows from using the upper bound of $\lambda^C_{0,p}(B)$ from Lemma \ref{lem:connection_upper_bound_balls} in the arguments of the proof of Theorem \ref{thm:volume}.
\end{proof}
\subsection{Proof of Theorem \ref{thm:conn_theorem}:}   Let $(M,g)$ be a closed Riemannian $n$-manifold such that
    \[
        |\operatorname{Ric}_M|\leq (n-1)\xi 
        \qquad\text{and}\qquad
        \operatorname{inj}_M\geq i_0.
    \]
     For $Q>1$, Theorem \ref{thm:harmonic_radius} provides a uniform lower bound for the \(C^{1,\alpha}\)-harmonic radius of $M$. For $x\in M$, $0<\varepsilon<r_H$, let $B=B(x,\varepsilon)$.  In this case, the first derivatives of the metric coefficients satisfy the estimate
\begin{equation*}
    \label{eq:partial_derivative_bound}
    \left|\partial g_{ij}\right|
\leq
\frac{Q-1}{r_H},
\end{equation*}
which yields a corresponding estimate for the Christoffel symbols. Therefore,  we have 
\[|\Gamma_{ij}^k|\le \frac{3nQ(Q-1)}{2r_H}, \qquad \text{ for all } 1\le i,j,k\le n.\]
We now follow the arguments in the proof of Lemma \ref{lem:connection_upper_bound_balls}.
Let $0<\varepsilon<r_H$ and for $x\in M$, let $B=B(x,\varepsilon)$ be a geodesic ball around $x$ of radius $\varepsilon$. Let $y^1, y^2,\cdots, y^n$ be harmonic coordinates on $B$ and $\omega_0=dy^1\wedge \cdots \wedge dy^p$. Then, \[|\triangledown\omega_0(x)|^2\le \frac{9n^2Q^{p+3}(Q-1)^2}{4r_H^2}.\]
Let $f:B\rightarrow \mathbb{R}$ be defined as $f(y)=1-\frac{1}{\varepsilon}d(x,y)$. Then for $\omega=f\omega_0$, \[\frac{\int_B|\triangledown\omega_0|^2 dv_g}{\int_B|\omega|^2dv_g}\le \frac{9n^2Q^{2p+3}(Q-1)^2}{r_H^2}\frac{V_{-\xi}(\varepsilon)}{V_{-\xi}\left(\frac{\varepsilon}{2}\right)}.\]
 Let $\lambda_{0,p}^C(B)$ denote the first non-zero Dirichlet eigenvalue of the connection Laplacian on $p$-forms in $B$. Then,
\[\lambda_{0,p}^C(B)\le \frac{\int_B|\triangledown\omega|^2 dv_g}{\int_B|\omega|^2dv_g}\le 9Q^{2p}
\left(\frac{1}{\varepsilon^2}+
\frac{2n^2Q^3(Q-1)^2}{r_H^2}\right)\frac{V_{-\xi}(\varepsilon)}{V_{-\xi}\left(\frac{\varepsilon}{2}\right)}.\]
 Following the arguments in Theorem \ref{thm:conn_ric_lower}, we get the eigenvalue estimates of the connection Laplacian.

\par

\section*{Acknowledgement}
 Anusha Bhattacharya is supported by the Council of Scientific and Industrial Research (CSIR).
 
\section*{Competing interests}

The authors have no conflict of interest and no financial interests in this article.
\section*{Ethics approval}
The submitted work is original and not submitted to more than one journal for simultaneous consideration.

\bibliographystyle{amsplain}
\bibliography{mybib}
\end{document}